\documentclass[final]{dmtcs-episciences}
\usepackage[utf8]{inputenc}
\usepackage{subfigure,amsmath,amsthm,amssymb,verbatim}

\theoremstyle{plain}
\newtheorem{theorem}{Theorem}[section]
\newtheorem{lemma}[theorem]{Lemma}
\newtheorem{corollary}[theorem]{Corollary}
\newtheorem{proposition}[theorem]{Proposition}

\DeclareMathOperator{\R}{\mathbb{R}}
\newcommand{\calP}{\ensuremath{\mathcal{P}}}
\newcommand{\calQ}{\ensuremath{\mathcal{Q}}}

\author{Jean Cardinal\affiliationmark{1}
  \and Michael S. Payne\affiliationmark{2}
  \and Noam Solomon\affiliationmark{3}}
\title[Ramsey-Type Theorems for Lines in 3-space]{Ramsey-Type Theorems for Lines in 3-space}
\affiliation{
  Universit\'e libre de Bruxelles (ULB), Belgium\\
  Monash University, Melbourne, Australia\\
  Tel-Aviv University, Israel}

\received{2016-1-26}
\accepted{2016-8-21}
\revised{2016-8-20}
\publicationdetails{18}{2016}{3}{14}{1367}

\begin{document}
\maketitle
\sloppy
\begin{abstract}
We prove geometric Ramsey-type statements on collections of lines in 3-space. These statements give guarantees on the size of a clique or an independent set in (hyper)graphs induced by incidence relations between lines, points, and reguli in 3-space. Among other things, we prove the following:
\begin{itemize}
\item The intersection graph of $n$ lines in $\R^3$ has a clique or independent set of size $\Omega(n^{1/3})$.
\item Every set of $n$ lines in $\R^3$ has a subset of $\sqrt n$ lines that are all stabbed by one line, or a subset of $\Omega\left(\left(n/\log n\right)^{1/5}\right)$ such that no $6$-subset is stabbed by one line.
\item Every set of $n$ lines in general position in $\R^3$ has a subset of $\Omega(n^{2/3})$ lines that all lie on a regulus, or a subset of $\Omega(n^{1/3})$ lines such that no $4$-subset is contained in a regulus.
\end{itemize}
The proofs of these statements all follow from geometric incidence bounds -- such as the Guth-Katz bound on point-line incidences in $\R^3$ -- combined with Tur\'an-type results on independent sets in sparse graphs and hypergraphs. As an intermediate step towards the third result, we also show that for a fixed family of plane algebraic curves with $s$ degrees of freedom, every set of $n$ points in the plane has a subset of $\Omega (n^{1-1/s})$ points incident to a single curve, or a subset of $\Omega (n^{1/s})$ points such that at most $s$ of them lie on a curve.
Although similar Ramsey-type statements can be proved using existing generic algebraic frameworks, the lower bounds we get are much larger than what can be obtained with these methods.
The proofs directly yield polynomial-time algorithms for finding subsets of the claimed size. 
\end{abstract}

\section{Introduction}

Ramsey theory studies the 
conditions under which particular discrete structures must contain certain substructures.
Ramsey's theorem says that for every $n$, every sufficiently large graph has either a clique or an independent set of size $n$. Early geometric Ramsey-type statements include the Happy Ending Problem on convex quadrilaterals in planar point sets, and the Erd\H{o}s-Szekeres Theorem on subsets in convex position~\cite{ES35}.

We prove a number of Ramsey-type statements involving lines in $\R^3$. The combinatorics of lines in space is a widely studied topic which arises in many applications such as computer graphics, motion planning, and solid modeling~\cite{CEGSS96}. Our proofs combine two main ingredients: geometric information in the form of bounds on the number of incidences among the objects, and a Tur\'an-type theorem that converts this information into a Ramsey-type statement. We establish a general lemma that allows us to streamline the proofs.

Ramsey's Theorem for graphs and hypergraphs only guarantees the existence of rather small cliques or independent sets. However, as discussed below, for the geometric relations we study the bounds are known to be much larger. Therefore we are interested in finding the correct asymptotics. In particular, we are interested in the \emph{Erd\H os-Hajnal property}. A class of graphs has this property if each member with $n$ vertices has either a clique or an independent set of size $n^{\delta}$ for some constant $\delta>0$. This comes from the {\em Erd\H{o}s-Hajnal conjecture} which states that, for each graph $H$, the family of graphs excluding $H$ as an induced subgraph has this property. Our results yield new Erd\H os-Hajnal exponents for each of the classes of (hyper)graphs studied. 

The results presented here make use of important recent advances in combinatorial geometry.
The key example is the bound on the number of incidences between points and lines in $\R^3$ given by Guth and Katz~\cite{GK} in their recent solution of the Erd\H os distinct distances problem.
Such results have sparked a lot of interest in the field, and it can be expected that further progress will yield further Ramsey-type results.

\subsection{A general framework}

In general we consider two classes of geometric objects $\calP$ and $\calQ$ in $\R^d$ and a binary
incidence relation contained in $\calP\times\calQ$. For a finite set $P\subseteq\calP$ and a fixed 
integer $t\geq 2$, we say that a $t$-subset $S\in {P\choose t}$ is {\em degenerate} whenever there exists
$q\in\calQ$ such that every $p\in S$ is incident to $q$. Hence the incidence
relation together with the integer $t$ induces a $t$-uniform hypergraph $H=(P,E)$, where
$E\subseteq {P\choose t}$ is the set of all degenerate $t$-subsets of $P$.
A clique in this hypergraph is a subset $S\subseteq P$ such that ${S\choose t}\subseteq E$.
Similarly, an independent set is a subset $S\subseteq P$ such that ${S\choose t}\cap E=\emptyset$.

In what follows, the families $\cal P$ and $\cal Q$ will mostly consist of lines or points in 3-space.
We are interested in Ramsey-type statements stating that the $t$-uniform hypergraph $H$
induced by a set $P\subset\calP$ of size $n$ has either a clique of size $\omega (n)$
or an independent set of size $\alpha (n)$.

\subsection{Previous results}

We first briefly survey some known results that fit into this framework.
In many cases, either $\calP$ or $\calQ$ is a set of points.
When $\calP$ is a set of points, finding a large independent set amounts to finding a large subset of points in some kind of general position defined with respect to $\calQ$.
When $\calQ$ is the set of points, we are dealing with intersections between the objects in $\calP$. In particular, the case $t=2$ corresponds to the study of geometric intersection graphs.

\subsubsection*{General position subset problems}

A set in $\R^d$ is usually said to be in general position whenever no $d+1$ points lie on a hyperplane.
For points and lines in the plane, Payne and Wood proved that the 
Erd\H{o}s-Hajnal property essentially holds with exponent $1/2$~\cite{PW13}. 
Cardinal et al. proved an analogous result in $\R^d$~\cite{CTW14}.
\begin{theorem}[\cite{PW13,CTW14}]
\label{thm:gp}
Fix $d\geq 2$. Every set of $n$ points in $\R^d$ contains $\sqrt{n}$ cohyperplanar points
or $\Omega ((n/\log n)^{1/d})$ points in general position.
\end{theorem}
In both cases, the proofs rely on incidence bounds, in particular the
Szemer\'edi-Trotter Theorem~\cite{ST83} in the plane, and the  
point-hyperplane incidence bounds due to Elekes and T\'oth~\cite{ET05} in $\R^d$. 
In this paper we formalise the technique used in those proofs in order to easily apply it to other incidence relations.

\subsubsection*{Erd\H{o}s-Hajnal properties for geometric intersection graphs}

A survey of Erd\H{o}s-Hajnal properties for geometric intersection graphs was
produced by Fox and Pach~\cite{FP08}. A general Ramsey-type statement for 
the case where $\calP$ is the set of plane convex sets has been known for a long time.
In what follows, a {\em vertically convex} set is a set whose intersection with any vertical line is a line segment.
\begin{theorem}[Larman et al.~\cite{LMPT94}]
Any family of $n$ compact, connected and vertically convex sets in the plane contains at least $n^{1/5}$ members that are either pairwise disjoint or pairwise intersecting.
\end{theorem}
Larman et al. also showed that there exist arrangements of $k^{2.3219}$ line segments with at most $k$
pairwise crossing and at most $k$ pairwise disjoint segments. This lower bound was improved successively by K\'arolyi et al.~\cite{KPT97}, and Kyncl~\cite{K12}.

More recently Fox and Pach studied intersection graphs of a large variety of other geometric objects~\cite{FP12}. 
For example they proved the following about families of $s$-intersecting curves in the plane -- families such that 
no two curves cross more than $s$ times.
\begin{theorem}[Fox-Pach~\cite{FP12}]
For each $\epsilon >0$ and positive integer $s$, there is $\delta = \delta(\epsilon, s)>0$
such that if $G$ is an intersection graph of a $s$-intersecting family of $n$ curves in the plane, 
then $G$ has a clique of size at least $n^{\delta}$ or an independent set of size at least $n^{1-\epsilon}$.
\end{theorem}

Erd\H{o}s-Hajnal properties for hypergraphs have been proved by Conlon, Fox, and Sudakov~\cite{CFS12}. 

\subsubsection*{Semi-algebraic sets and relations}

A very general version of the problem for the case $t=2$ has been studied by Alon et al.~\cite{APPRS05}.
Here Ramsey-type results are provided for intersection relations between semialgebraic sets of constant description
complexity in $\R^d$. It was shown that intersection graphs of such objects always have the Erd\H os-Hajnal property. The proof
combines a linearisation technique with a space decomposition theorem due to Yao and Yao~\cite{YY85}.
The following general statement can be extracted from their proof.
\begin{theorem}
  \label{thm:alg}
  Consider a relation $R$ on elements of a family $\cal F$ of semi-algebraic sets of constant description complexity.
  Suppose that each element $f\in {\cal F}$ can be parameterized by a point $f^*\in\R^d$, and that the relation $R$
  can be mapped into a semi-algebraic set $R^*$ in $\R^{2d}$. For each $g\in {\cal F}$, let $\Sigma_g = \{f^*\in \R^d : (f^*, g^*)\in R^*\}$.
  Let $Q$ be the smallest dimension of a space $\R^Q$ in which the description of $\Sigma_g$ becomes linear, and let $k$ be the number of
  bilinear inequalities in the definition of $R^*$ in $\R^Q$. Then the graph of the relation $R$ satisfies the Erd\H{o}s-Hajnal property
  with exponent $1/(2k(Q+1))$.
\end{theorem}
A similar result is given for the so-called {\em strong} version of the Erd\H os-Hajnal property: for every such intersection relation, there
exists a constant $\epsilon$ and a pair of subfamilies ${\cal F}_1, {\cal F}_2\subseteq {\cal F}$, each of size at least $\epsilon |{\cal F}|$,
such that either every element of ${\cal F}_1$ intersects every element of ${\cal F}_2$, or no element of ${\cal F}_1$ intersects any
element of ${\cal F}_2$. The exponent for the usual Erd\H os-Hajnal statement is a function of this $\epsilon$.

As an example, Alon et al. applied their machinery to prove the following result on arrangement of lines in $\R^3$.
\begin{theorem}[Alon et al.~\cite{APPRS05}]
Every family of $n$ pairwise skew lines in $\R^3$ contains at least $k\geq n^{1/6}$ 
elements $\ell_1,\ell_2,\ldots ,\ell_k$ such that $\ell_i$ passes above $\ell_j$ for all $i<j$. 
\end{theorem}
For the problems we consider, however, the exponents we obtain are significantly larger than what can be obtained from Theorem~\ref{thm:alg}.

A general version of this problem in which degenerate $t$-tuples are defined by a finite number of polynomial equations and inequalities of bounded
description complexity has recently been studied by Conlon et al.~\cite{CFPSS13}.
 They show that the Ramsey numbers in this general setting grow like towers of height $t-1$, and that this is asymptotically tight. Such a setting is relevant here, since we also consider Erd\H os-Hajnal statements for some geometric hypergraphs.

\subsection{Summary of our results}

In Section~\ref{sec:prelim} we give a simple lemma that allows to convert geometric incidence bounds into bounds on the
number of degenerate subsets, hence on the number of hyperedges of the hypergraphs of interest. We also recall the
statements of the Tur\'an bound for hypergraphs due to Spencer.

Section~\ref{sec:pointslinesR3} deals with the case where $\calP$ and $\calQ$ are lines and points in $\R^3$.
A natural object to consider is the intersection graph of lines in $\R^3$, for which we prove the Erd\H{o}s-Hajnal
property with exponent $1/3$.
\begingroup
\def\thetheorem{\ref{thm:igl}}
\begin{theorem}
The intersection graph of $n$ lines in $\R^3$ has a clique or independent set of size $\Omega(n^{1/3})$. 
\end{theorem}
\addtocounter{theorem}{-1}
\endgroup
This makes use of the Guth-Katz incidence bound between points and lines in $\R^3$~\cite{GK2}.
We further show that this exponent can be raised to $1/2$ if we consider lines in the projective 3-space.
We also show how to obtain bounds on the size of independent sets for $t=3$, in which a subset of lines in general position is
defined as a set of lines with no three intersecting in the same point.

Section~\ref{sec:lineslinesR3} deals with the setting where both $\calP$ and $\calQ$ are lines in $\R^3$.
We prove the following theorem.
\begingroup
\def\thetheorem{\ref{thm:main3}}
\begin{theorem}
Let $L$ be a set of $n$ lines in $\R^3$. Then
either there is a subset of $\sqrt n$ lines of $L$ that are all
stabbed by one line, or there is a subset of
$\Omega\left(\left(n/\log n\right)^{1/5}\right)$ lines of $L$ such
that no $6$-subset is stabbed by one line.
\end {theorem}
\addtocounter{theorem}{-1}
\endgroup
The proof involves lifting the set of lines to a set of points and hyperplanes in $\R^5$, and applying
the Ramsey-type result on points and hyperplanes due to Cardinal et al.~\cite{CTW14}.
The latter in turn relies on a point-hyperplane incidence bound due to Elekes and T\'oth~\cite{ET05}.

Finally, in Section~\ref{sec:linesreguliR3} we introduce the notion of a subset of lines in general position in $\R^3$
with respect to reguli, defined as loci of lines intersecting three pairwise skew lines. We use the Pach-Sharir
bound on incidences between points and curves in the plane~\cite{PS98} to obtain the following result.
\begingroup
\def\thetheorem{\ref{4inreg}}
\begin{theorem}
Let $L$ be a set of $n$ pairwise skew lines in $\R^3$. Then there are $\Omega (n^{2/3})$ lines on a regulus, or $\Omega (n^{1/3})$ lines, no 4-subset of which lie on a regulus.  
\end{theorem}
\addtocounter{theorem}{-1}
\endgroup
We also explain how to use a line-regulus incidence bound due to Aronov et al.~\cite{AKS05} for an alternative proof of this result.

The large subsets whose existence our results guarantee can be found in polynomial time. 
In each case, a degenerate $t$-subset is incident to only one element of $\calQ$
(for example, three collinear points lie on only one line). Furthermore, the cliques given by our results are of a particular type:
all the elements intersect a single element of $\calQ$ (for example, a collinear set of points).
Thus the largest such clique in the hypergraph $H$ can be found in polynomial time by checking all the elements of $\calQ$ that
determine a degenerate $t$-subset (for example, all lines determined by the point set).
If the clique size is small, Tur\'an-type theorems yield an independent set of a guaranteed minimum size.
These theorems are constructive, hence the large independent set can be found efficiently.

\section{Preliminaries}
\label{sec:prelim}

In order to prove the existence of large independent sets in hypergraphs with no
large clique, we proceed in two steps. First, we use incidence bounds to get
upper bounds on the density of the (hyper)graph. Then we apply Tur\'an's Theorem or its hypergraph analogue
to find a lower bound on the size of the independent set. This is an extension of the method used
to prove Theorem~\ref{thm:gp} in~\cite{PW13,CTW14}. The use of incidence bounds is also reminiscent
from the technique used by Pach and Sharir for the repeated angle problem~\cite{pachsharir}.

The following lemma will allow us to quickly convert incidence bounds into density conditions. 
Recall that we consider two families $\calP$ and $\calQ$ with an incidence relation in $\calP\times\calQ$,
and that a $t$-subset $S$ of $\calP$ is said to be degenerate whenever there exists $q\in\calQ$ such that 
every $p\in S$ is incident to $q$.

\begin{lemma}\label{lem:nitoedges}
Let $P$ be a subset of $\calP$ with $|P|=n$, such that no element
of $\calQ$ is incident to more than $\ell$ elements of $P$.
Let us denote by $P_{\geq k}$ the number of elements of $\calQ$ incident to at least $k$ elements of $P$,
and suppose $P_{\geq k}\lesssim g(n)/k^a$ for some function $g$ and some real number $a$. Then the number of
degenerate $t$-subsets induced by $P$ is at most
$$
m\lesssim
\begin{cases}
g(n) & \mathrm{\ if\ } t < a, \\
g(n)\log\ell & \mathrm{\ if\ } t = a, \\
g(n)\ell^{t-a} & \mathrm{\ if\ } t > a.
\end{cases}
$$
\end{lemma}
\begin{proof}
Let $P_j$ be the number of elements of $\calQ$ incident to {\em exactly} $j$ elements of $P$.
Then
\begin{eqnarray*}
m & = & \sum_{j=t}^{\ell} P_j {j\choose t} < \sum_{j=1}^{\ell} P_j j^t < \sum_{j=1}^{\ell} P_j \left( t \sum_{k=1}^{j}k^{t-1}\right) 
 \simeq \sum_{k=1}^{\ell} k^{t-1} \left( \sum_{j=k}^{\ell} P_j\right) \\
  & = & \sum_{k=1}^{\ell} k^{t-1} P_{\geq k} \lesssim  g(n) \sum_{k=1}^{\ell} k^{t-1-a},
\end{eqnarray*}
where we use that $\sum_{k=1}^j k^{t-1}=j^t/t+O(j^{t-1})$, and $t=O(1)$.
The final sum simplifies differently depending on the relative values of $t$ and $a$.
\end{proof}

We recall the statement of Tur\'{a}n's Theorem.
\begin{theorem}[Tur\'{a}n~\cite{turan}]\label{thm:turan} 
Let $G$ be a graph with $n$ vertices and $m$ edges. Then $\alpha (G) \geq \frac{n}{\frac{2m}{n} + 1}$.
Thus if $m<n/2$ then $\alpha(G) > n/2$. Otherwise $\alpha(G) \geq n^2 /4m$.
\end{theorem}
The hypergraph version of this result was proved by Spencer.
\begin{theorem}[Spencer~\cite{S72}]\label{thm:spencer}
Let $H$ be a $t$-uniform hypergraph with $n$ vertices and $m$ edges. If $m<n/t$
then $\alpha (H)>n/2$. Otherwise
$$
\alpha (H) \geq \frac{t-1}{t^{t/(t-1)}} \frac n{(m/n)^{1/(t-1)}}.
$$
\end{theorem}

\section{Points and lines in $\R^3$}
\label{sec:pointslinesR3}

The recent resolution of Erd\H{o}s' distinct distance problem by Guth and Katz involves new 
bounds on the number of incidences between points and lines in $\R^3$~\cite{GK}. This breakthrough has
fostered research on point-line incidence bounds in space. In this section and the next, 
we exploit those recent results to obtain various new Ramsey-type statements on point-line incidence 
relations in space.

\subsection{General position with respect to lines}

Theorem~\ref{thm:gp} for $d=2$ states that in a set $P$ of
$n$ points in the plane there exist either $\sqrt{n}$ collinear points,
or $\Omega(\sqrt{n/ \log n})$ points with no three collinear. 
Payne and Wood~\cite{PW13} conjectured that the true size
should be $\Omega(\sqrt{n})$, but this small improvement has proven
elusive.

Here we consider the same question but with $\calP = \R^3$, 
$\calQ$ defined as the set of lines in $\R^3$, and $t=3$. 
Hence we consider that a set $P\subset\R^3$ is in general position when no 
three points are collinear.
So far this is the same question as in the planar case, since a
point set in higher dimensional space can always be projected to the
plane in a way that maintains the collinearity relation. However,
under a small extra assumption, namely that among the $n$ points in
$\R^3$, at most $n / \log n$ are coplanar, we are able to remove
the $\log n $ factor in the independent set. This sheds some light
on the nature of potential counterexamples to the conjecture of
Payne and Wood.

We will use the following result of Dvir and Gopi~\cite{DG15}, which is deduced from Guth and Katz~\cite{GK2}.
\begin{theorem}
  \label{thm:PkR3}
Given a set $P$ of $n$ points in $\R^3$, such that at most $s$ points are contained in a plane, the number $P_{\geq k}$ of lines containing at least $k$ points is $$ P_{\geq k} \lesssim \frac{n^2}{k^4}+\frac{ns}{k^3}+\frac{n}{k} .$$
\end{theorem}

\begin{theorem}
\label{thm:gpR3}
Any set of $n$ points in $\R^3$ such that at most $n / \log n $ of the points lie in a 
plane contains either $\sqrt{n}$ collinear points or $\Omega(\sqrt{n})$ with no three collinear.
\end{theorem}
\begin{proof}
  We apply Lemma~\ref{lem:nitoedges} on each term of the bound in Theorem~\ref{thm:PkR3}. We obtain that the number of degenerate 3-subsets of points is
  $$
  m \lesssim n^2 + ns \log \ell + n\ell^2,
  $$
  where $\ell = \sqrt{n}$ and $s=n/ \log n$. Hence the dominating term is $n^2$. Applying Theorem~\ref{thm:spencer} yields an independent set of size $\Omega(\sqrt{n})$.
\end{proof}

In fact, this theorem holds in $\R^d$ for $d>3$. To see this, we take a generic projection of $\R^d$ onto $\R^3$.
The condition that at most $n/\log n$ lines are coplanar remains true under a generic projection. 

\subsection{Line intersection graphs in $\R^3$}

We now consider the setting in which the family $\calP$ is the set of lines in $\R^3$ and
$\calQ = \R^3$. The first subcase we consider is $t=2$, or in other words, intersection graphs of lines.
Note that in an intersection graph of lines in $\R^3$,
every clique of size $k\geq 2$ corresponds either to a subset of $k$ lines
having a common intersection point, or to a subset of $k$ lines
lying in a plane. However, $k$ lines lying in a plane do not form a clique
if some of them are parallel.

We consider a set $L$ of $n$ lines in $\R^3$, such that no
more than $\ell$ lines intersect in a point, and at most $s$ lines lie
in a common plane or a \emph{regulus}. We recall that a regulus is a
degree two algebraic surface, which is the union of all the lines in
$\R^3$ that intersect three pairwise-skew lines in $\R^3$. It is a
\emph{doubly-ruled} surface; each point on a regulus is incident to
precisely two lines fully contained in the regulus. Moreover, there
are two \emph{rulings} for the regulus; every line from one ruling
intersects every line from the other ruling, and does not intersect
any line from the same ruling.

We first recall two important theorems of Guth and Katz~\cite{GK2}.
In what follows, $P_{\ge k}$ denotes the number of points incident to at least $k$ lines in $L$.

\begin{theorem} [\protect{\cite[Theorem 4.5]{GK2}}]
\label{th:gk21} If $L$ is a set of $n$ lines, so that no plane contains more than $s$ lines, then for $k \ge 3$ we have
$$
P_{\ge k} \lesssim \frac {n^{3/2}} {k^2} +
\frac {n s} {k^3} + \frac n k.
$$
\end{theorem}

\begin{theorem}[\protect{\cite[Theorem 2.11]{GK2},\cite{SS3d}}]
\label{th:gk22} If $L$ is a set of $n$ lines, so that no plane or regulus contains more than $s$ lines, then $P_{\ge 2} \lesssim
n^{3/2} + ns$.
\end{theorem}
Note the difference between the two statements: the assumption that no regulus contains
more than $s$ lines is required for the case $k=2$ only.

Applying Lemma~\ref{lem:nitoedges} to the bounds in Theorems~\ref{th:gk21} and~\ref{th:gk22} yields the following.
\begin {proposition}
\label{pr:gk2} Given a set $L$ of $n$ lines, so that no plane or
regulus contains more than $s$ lines, and no point is incident to
more than $\ell$ lines of $L$, the number of line-line incidences is
$O(n^{3/2}\log \ell + ns + n\ell)$.
\end {proposition}

\begin{lemma}
\label{lem:iglines3}
Consider a set $L$ of $n$ lines in $\R^3$, such that no plane contains more than $s$ lines, and no point is incident to
more than $\ell$ lines of $L$.
Let $G$ be the intersection graph $L$. 
If $s,\ell \lesssim n^{1/2}$, then $\alpha (G) \gtrsim \sqrt{n} / \log \ell$.
Moreover, if $r:=\max\{s,\ell\} \gtrsim n^{\frac{1}{2}+\epsilon}$ for some $\epsilon >0$, then $\alpha (G) \gtrsim n / r$.
\end{lemma}
\begin{proof}
If there is some regulus containing at least $n^{1/2}$ lines, we
divide the lines into the two rulings of the regulus. One
ruling contains at least half the lines, and as the lines in
one ruling do not intersect one another, it follows that
$\alpha(G)\gtrsim n^{1/2}$. We may therefore assume that the number of
lines contained in a common regulus is at most $n^{1/2}$.

If $s,\ell \leq n^{1/2}$, the first term in the bound in Proposition~\ref{pr:gk2} dominates, and applying Theorem~\ref{thm:turan} gives
$\alpha (G) \gtrsim \sqrt{n} / \log \ell$. If $r\geq n^{\frac{1}{2}+\epsilon}$, one of the latter terms dominates, and we apply Theorem~\ref{thm:turan} to get 
$\alpha(G)\gtrsim n/r$.
\end{proof}

\begin{theorem}
\label{thm:igl}
The intersection graph of $n$ lines in $\R^3$ has a clique or independent set of size $\Omega(n^{1/3})$. 
\end{theorem}
\begin{proof}
Suppose that such a graph $G$ has $\alpha(G) \ll n^{1/3}$. Then by
Lemma~\ref{lem:iglines3}, $\max\{s,\ell\}\gtrsim n^{2/3}$. If $\ell
\gtrsim n^{2/3}$ we are done, so $s \gtrsim n^{2/3}$. Therefore, we may
assume that there is a plane containing $n^{2/3}$ lines. 
Divide these lines into classes of pairwise parallel lines. 
If some class contains at least $n^{1/3}$ lines, we have $\alpha(G) \gtrsim n^{1/3}$.
Otherwise, there are at least $n^{1/3}$ different parallel classes.
Choosing one line from each class yields a clique of size $n^{1/3}$. 
\end{proof}

Note that the Erd\H os-Hajnal property for intersection graphs of lines in $\R^3$ can be directly established from
Theorem~\ref{thm:alg} by Alon et al.~\cite{APPRS05}, but with a much smaller exponent. In their setting, we can represent
the intersection relation between lines using Pl\"ucker coordinates in $\R^5$, and using two inequalities. This yields
$k=2$ and $Q=5$, and an Erd\H os-Hajnal exponent of $1/24$. Although it is likely that it can be improved by shortcutting steps in the general proof,
any exponent we would get would still be far from 1/3. 

We now make a connection with intersection graphs of lines in space and line graphs. Recall that the line graph of a graph $G$
has the set of edges $E(G)$ as vertex set, and an edge between two edges of $G$ whenever they are incident to the same vertex of $G$.
Observe that for every graph $G$, the line graph of $G$ can be represented as the intersection graph of lines in $\R^3$ by drawing $G$ on
  a vertex set in general enough position in $\R^3$, and extending the edges of the drawing to lines.
By applying Vizing's Theorem, which says that the edge chromatic number of every graph is at most $\Delta +1$, we may see that the class 
of line graphs has the Erd\"os--Hajnal property with exponent $1/2$.
The question of the exact Erd\"os--Hajnal exponent for intersection graphs of lines in $\R^3$ remains open -- it lies somewhere between $1/3$ and $1/2$.

Finally we note that for sets of lines in projective space, coplanar sets of lines always form a clique. The following stronger result can be directly obtained.
\begin{theorem}\label{thm:projlineint}
For every intersection graph $G$ of $n$ lines in $\mathbb{P}^3$, either $\omega (G) \geq \sqrt{n}$ or $\alpha (G) = \Omega (\sqrt{n} / \log n)$.
\end{theorem}
Hence intersection graphs of lines in the projective plane satisfy the Erd\H{o}s-Hajnal property with exponent roughly $1/2$.

\subsection{Independent Sets of Lines for $t=3$}

We now consider the case in which $\calP$ is the set of lines in $\R^3$, $\calQ=\R^3$ and $t=3$. 
This can be seen as a kind of three-dimensional version of the dual of the result of Payne and Wood (Theorem~\ref{thm:gp} with $d=2$).
\begin{theorem}
\label{thm:iglines3b}
Consider a collection $L$ of $n$ lines in $\R^3$, such that at most $s$ lie in a plane, with $s \leq n / \log n$.
Then there exists a point incident to $\sqrt{n}$ lines, or a subset of $\Omega (\sqrt{n})$ lines 
such that at most two intersect in one point.
\end{theorem}
\begin{proof}
We let $\ell$ be the largest number of lines intersecting in one point, and suppose $\ell < \sqrt{n}$.
Applying Lemma~\ref{lem:nitoedges} and Theorem~\ref{th:gk21}, we get that the number of triples sharing a point is
at most $$ m \lesssim \ell n^{3/2} + ns \log \ell +n \ell^2 \lesssim n^2 .$$
Then by Theorem~\ref{thm:spencer} we have an independent set of size $\Omega(\sqrt{n})$.
\end{proof}

If the above theorem is stated with dependence on $\ell$, we get $\Omega(n^{3/4}/\sqrt{\ell})$.
If $s$ is allowed to be as large as $n$, we are back in the dual of general position subset selection, and we get $\Omega(\sqrt{n/ \log n})$, the same as Theorem~\ref{thm:gp}.

\section {Stabbing lines in $\R^3$}
\label{sec:lineslinesR3}

Three lines in $\R^3$ are typically intersected by a fourth line, except in certain degenerate cases.
Thus it makes sense to study configurations of lines in $\R^3$, and to consider a set of $4$ or more lines degenerate if all its elements are intersected by another line. 
Here we provide a result for $6$-tuples of lines. 

We define a 6-tuple of lines to be degenerate if all six lines are intersected (or ``stabbed'') by a single line in $\R^3$. 
We call this line a \emph{stabbing line} for the $6$-tuple of lines.
So in our framework this is the setting in which both $\cal P$ and $\cal Q$ are the set of lines in $\R^3$, and $t=6$. 

We make use of the Pl\"ucker coordinates and coefficients for lines in $\R^3$, which are a common tool for dealing with incidences between lines, 
see e.g.~Sharir~\cite{Sharir}. 
Let $a=(a_0:a_1:a_2:a_3), \
b=(b_0:b_1:b_2:b_3)$ be two points on a line $\ell$, given in
projective coordinates. By definition, the Pl\"ucker coordinates of
$\ell$ are given by
$$(\pi_{01}: \pi_{02}: \pi_{12}: \pi_{03}: \pi_{13}: \pi_{23}) \in
\mathbb P^5,$$ where $\pi_{ij}=a_i b_j - a_j b_i$ for $0\le i< j \le 3$.
Similarly, the Pl\"ucker coefficients of $\ell$ are given by
$$(\pi_{23}: -\pi_{13}: \pi_{03}: \pi_{12}: -\pi_{02}: \pi_{01}) \in
\mathbb P^5,$$ i.e., these are the Pl\"ucker coordinates written in
reverse order with two signs flipped. The important property is that
two lines $\ell_1$ and $\ell_2$ are incident if and only if the
Pl\"ucker coordinates of $\ell_1$ lie on the hyperplane defined by
the Pl\"ucker coefficients of $\ell_2$ and vice versa. Therefore, we
define $\tilde {\cal P}$, and $\tilde {\cal Q}$ to be the points in
$\mathbb P^5$ defined by the Pl\"ucker coordinates of the lines in
$L$, and the hyperplanes defined by the Pl\"ucker coefficients of
the lines in $\R^3$, respectively. The incidence relation between
points in $\tilde {\cal P}$ and hyperplanes in $\tilde {\cal Q}$ is
the standard incidence relation between points and hyperplanes. The
integer $t$ is set to 6, and a 6-tuple of points in $\tilde
{\cal P}$ is degenerate whenever there is a hyperplane in $\tilde
{\cal Q}$ which is incident to all six points in the 6-tuple.

We prove the following Ramsey-type result for stabbing lines in $\R^3$.
\begin {theorem}
  \label{thm:main3}
  Let $L$ be a set of $n$ lines in $\R^3$. Then
either there is a subset of $\sqrt n$ lines of $L$ that are all
stabbed by one line, or there is a subset of
$\Omega\left(\left(n/\log n\right)^{1/5}\right)$ lines of $L$ such
that no $6$-subset is stabbed by one line.
\end {theorem}

Theorem~\ref{thm:main3} is an immediate consequence of the following generalisation of Theorem~\ref{thm:gp}. 
The difference is that the set of hyperplanes $\cal H$ is arbitrary instead of being the set of all hyperplanes in $\R^d$.
\begin {theorem}
\label{th:main1} Let $\cal H$ be a set of hyperplanes in $\R^d$.
Then, every set of $n$ points in $\R^d$ with at most $\ell$ points on any hyperplane in $\cal H$, where $\ell=O(n^{1/2})$, contains a
subset of $\Omega\left(\left(n/\log \ell\right)^{1/d}\right)$ points so
that every hyperplane in $\cal H$ contains at most $d$ of these
points.
\end {theorem}

Theorem~\ref{th:main1}, with $d=5$, applied to the points and hyperplanes given by the Pl\"ucker coordinates and coefficients, implies Theorem~\ref{thm:main3}.
Theorem~\ref{th:main1} follows from the following generalized version of Lemma 4.5 of Cardinal et al.~\cite{CTW14}.
\begin {lemma}
\label {le:mainle} Fix $d \ge 2$ and a set $\cal H$ of hyperplanes in $\R^d$. Let $P$ be a set of $n$ points in
$\R^d$ with no more than $l$ points in a hyperplane in $\cal H$, for some $l=O(n^{1/2})$. Then, the number of $(d+1)$-tuples in $P$ that lie
in a hyperplane in $\cal H$ is $O(n^d \log l)$.
\end {lemma}
The difference between this lemma and the original version in~\cite{CTW14} is that the set of hyperplanes $\cal H$ is arbitrary, rather than being the set of all hyperplanes.  
The proof is similar to that of Cardinal et al., and is given in Appendix~\ref{app:cardproof}.

The following result provides a simple upper bound.
\begin{theorem}
For every constant integer $t\geq4$, there exists an arrangement $L$ of $n$ lines in $\R^3$ such that there is no subset of more than $O(\sqrt n)$ lines that are all stabbed by one line, nor any subset of more than $O(\sqrt n)$ lines with no $t$ stabbed by one line.
\end{theorem}
\begin{proof}
Construct $L$ as follows: pick $\sqrt n$ parallel planes, each containing $\sqrt n$ lines, with no three intersecting and no two parallel.
Consider a subset stabbed by one line. Either it has three coplanar lines; then it must be fully contained in one of the planes and contains at most $\sqrt n$ lines;
or it has no three coplanar lines, hence contains at most two lines from each plane, and has at most $2\sqrt n$ lines.
Now consider a subset such that no $t$ lines are stabbed by one. Then it contains at most $t-1$ lines from each plane, and has at most $(t-1)\sqrt n$ lines.
\end{proof}

\section{Lines and reguli in $\R^3$}
\label{sec:linesreguliR3}

Consider the case in which $\calP$ is the class of lines in $\R^3$, $\calQ$ is the class of reguli, and $t=4$.
Let $P$ be a set of $n$ lines, and assume that the lines in $P$ are pairwise skew.
Every triple of lines in $P$ therefore determines a single regulus, and we may consider the set of all reguli determined by $P$.
We consider the containment relation rather than intersection -- we are interested in $4$-tuples that all lie in the same regulus.
In order to prove our result, we first reformulate previously known incidence bounds between points and curves in the plane.

\subsection{General position with respect to algebraic curves}

We first consider the case where $\calP = \R^2$ and $\calQ$ is a family of algebraic curves of bounded degree.
We define the number of degrees of freedom of a family of algebraic curves $\cal C$ to be the minimum value $s$ such that 
for any $s$ points in $\R^2$ there are at most $c$ curves passing through all of them, for some constant $c$.
Moreover, $\cal C$ has multiplicity type $r$ if any two curves in $\cal C$ intersect in at most $r$ points.
We consider a set of points to be in general position with respect to $\cal C$ when no $s+1$ points lie on a curve in $\cal C$.

It is possible to extract Ramsey-type statements for this situation directly from Theorem~\ref{thm:gp} via linearisation. 
For example, let us consider the special case of circles, where $s=3$. Given a set
of points in the plane, we can lift it onto a paraboloid in $\R^3$ in such a way that a subset of the original set lies on a circle
(possibly degenerated into a line)
if and only if the corresponding lifted points lie on a hyperplane in $\R^3$. By applying Theorem~\ref{thm:gp} on the lifted set, we can show that 
there exists a subset of $\sqrt{n}$ points incident to a circle, or a subset of $\Omega ((n/\log n)^{1/3})$ points such that at 
most three of them lie on a circle. We show how we can improve on this.

In order to apply our technique, we need Szemer\'edi-Trotter-type bounds on the number of incidences 
between points and curves. This has been studied by Pach and Sharir~\cite{PS98}.
\begin{theorem}[\cite{PS98}]
\label{thm:curveST}
Let $P$ be a set of $n$ points in the plane and let $\mathcal{C}$ be a set of $m$ bounded degree plane algebraic curves with $s$
degrees of freedom and multiplicity type $r$. Then the number of point-curve incidences is at most
$$
I(P, \mathcal{C}) \leq C(r,s) \left( n^{s/(2s-1)}m^{(2s-2)/(2s-1)} + n + m \right)
$$
where $C(r,s)$ is a constant depending only on $r$ and $s$.
\end{theorem}
Pach and Sharir proved Theorem~\ref{thm:curveST} for simple curves with $s$ degrees of freedom and multiplicity type $r$. 
It is well known that one may replace simple curves with bounded degree algebraic curves, since such curves may be cut into a constant number of simple pieces.
Note that a set of bounded degree algebraic curves has constant multiplicity type if no two curves share a common component. 
Wang et al.~\cite{WangEtal} recently proved another result for incidences between points and algebraic curves, though for our purposes Theorem~\ref{thm:curveST} is stronger.

\begin{theorem}
\label{thm:gpcurve}
Consider a family $\mathcal{C}$ of bounded degree algebraic curves in $\R^2$ with constant multiplicity type and $s$ degrees of freedom, for some $s>2$. 
Then in any set of $n$ points in $\R^2$,
there exists a subset of $\Omega (n^{1-1/s})$ points incident to a single curve of $\mathcal{C}$, or
a subset of $\Omega (n^{1/s})$ points such that at most $s$ of them lie on a curve of $\mathcal{C}$.
\end{theorem}
\begin{proof}
Set $t=s+1$ and count the number of degenerate $t$-subsets.
We denote by $P_{\geq k}$ the number of curves of $\mathcal{C}$ containing at least $k$ points of $P$.
A direct corollary of Theorem~\ref{thm:curveST} is that, for values of $k$ larger than some constant,
$$
P_{\geq k} \lesssim \frac{n^s}{k^{2s-1}} + \frac nk.
$$
On the other hand, for smaller values of $k$, the trivial bound $P_{\geq k} \lesssim n^s$ holds since for any $s$ points, there are at most a constant number of curves passing through all of them.
Suppose now that no curve contains more than $\ell\lesssim n^{1-1/s}$ points of $P$. 
Since $s>2$, it follows that $t< 2s-1$.
Using Lemma~\ref{lem:nitoedges}, we deduce that the number of degenerate $t$-subsets is
$$
m\lesssim n^s + n\ell^s \lesssim n^s.
$$
Thus by Theorem~\ref{thm:spencer} there exists an independent set of size at least
$$
\frac{t-1}{t^{t/(t-1)}} \frac n {(m/n)^{1/(t-1)}} =  \Omega(n^{1/s}).
$$
\end{proof}

As an example, we can instantiate the result as follows for circles in the plane.
\begin{corollary}
\label{thm:gpcircle}
In any set of $n$ points in $\R^2$, there exists a subset of $\Omega (n^{2/3})$ points incident to a circle, or
a subset of $\Omega (n^{1/3})$ points such that no four of them lie on a circle.
\end{corollary}

Using the standard point-line duality, Theorem~\ref{thm:gp} states that for every arrangement of $n$ lines in $\R^2$, 
either there exists a point contained in $\sqrt{n}$ lines, or there exists a set of $\Omega((n/\log n)^{1/2})$ lines inducing a 
simple arrangement, that is, such that no point belongs to more than two lines. 
We provide a similar dual version of Theorem~\ref{thm:gpcurve}.
This corresponds to the case where $\calP$ is a family of algebraic curves with $s$ degrees of freedom, $\calQ = \R^2$, and $t=3$.
As mentioned before, the case $t=2$, or intersection graphs, has been studied previously~\cite{FP08,FP12}.
The proof is very similar to that of Theorem~\ref{thm:gpcurve} and omitted.
\begin{theorem}
\label{thm:dualgpcurve}
Consider a family $\mathcal{C}$ of bounded degree algebraic curves in $\R^2$ with constant multiplicity type and $s$ degrees of freedom, for some $s>2$. 
Then in any arrangement $C$ of $m$ such curves, there exists a subset of $\Omega (m^{1-1/s})$ curves intersecting in one point, or
a subset of $\Omega (m^{1/s})$ curves inducing a simple subarrangement, that is, such that at most two intersect in one point.
\end{theorem}

\subsection{Ramsey-type results for lines and reguli in $\R^3$}

We now come back to our original problem in which $\calP$ is the class of lines in $\R^3$, $\calQ$ is the class of reguli, and $t=4$. 
Here we restrict the finite arrangement $P\subset \calP$ to be pairwise skew, that is, pairwise nonintersecting and nonparallel.
We also consider the containment relation, that is, $\ell \in P$ is incident to $R\in Q$ if it is fully contained in it.

Recall that a regulus can be defined as a quadratic ruled surface which is the locus of all lines that are incident to three lines in general position. This surface is a \emph{doubly ruled} surface, that is, every point on a regulus is incident to precisely two lines fully contained in it. There are only two kinds of reguli, both of which are quadrics -- hyperbolic paraboloids and hyperboloids of one sheet; see for instance Sharir and Solomon~\cite{SS3dv} for more details.

\begin{theorem}\label{4inreg}
Let $L$ be a set of $n$ pairwise skew lines in $\R^3$. Then there are $\Omega (n^{2/3})$ lines on a regulus, or $\Omega (n^{1/3})$ lines, no 4-subset of which lie on a regulus.  
\end{theorem}
\begin{proof}
We map the lines in $L$ to a set $P$ of points in $\R^4$. 
This can be done for instance by associating with each line the $x$- and $y$-coordinates of the two points of intersection with the planes $z=0$ and $z=1$. (We may assume no line is parallel to these planes). 
Under this mapping, a ruling of a regulus corresponds to an algebraic curve in $\R^4$.
Let $C$ be the finite set of all curves corresponding to a ruling of a regulus determined by three lines in $L$. 
Note that \emph{any} triple of points in $\R^4$ is contained in at most one such curve, because three lines in $\R^3$ lie in at most one ruling of one regulus. (A pair of parallel or intersecting lines are not contained in a ruling of any regulus, even though they are contained in many reguli). 

Apply a generic projection $\pi$ from $\R^4$ to $\R^2$, and consider the arrangement of points $P'=\pi(P)$ together with the set of projected curves $C'= \pi(C)$.
Such a projection preserves the incidences between points and curves in $\R^4$, and only creates new intersections between pairs of curves (i.e.~`simple' crossings).
Three or more curves in $C'$ intersect in a point if and only if their preimages in $C$ intersect in a point.

The set of curves $C'$ has three degrees of freedom, since for any three points in $\R^2$ there are at most two curves passing through all of them.
Otherwise, if three curves pass through three points, the corresponding curves in $C$ also intersect in three points in $\R^4$, a contradiction. 

Moreover, the curves in $C'$ are algebraic of bounded degree, do not share common components, and thus have constant multiplicity type.
Applying Theorem~\ref{thm:gpcurve} with $s=3$, we obtain that there are $\Omega (n^{2/3})$ points of $\pi(P)$ on one curve, or $\Omega (n^{1/3})$ points of $\pi(P)$, no four of which lie on a curve. The result follows.  
\end{proof}

The bounds can be shown to be tight in the following sense.
\begin{theorem}
There exists a set $P$ of $n$ pairwise skew lines in $\R^3$ such that there is no subset of more than $O(n^{2/3})$ lines on a regulus, and no more than $O(n^{1/3})$ lines such that no 4-subset lie on a regulus.  
\end{theorem}
\begin{proof}
The set $P$ is constructed as follows: take a set of $n^{1/3}$ distinct reguli, and for each regulus take $n^{2/3}$ lines in one of its rulings, giving $n$ pairwise skew lines. 
Consider a subset of $P$ contained in a regulus. Either it is one of the chosen reguli, and it contains at most $n^{2/3}$ lines, or it contains at most two lines from each regulus, and has size at most $2n^{1/3}$. On the other hand, consider a subset of lines with no four on a regulus. It can contain at most three lines from each chosen regulus, and therefore has size at most $3n^{1/3}$.
\end{proof}

\paragraph{Alternative proof.} Aronov et al.~\cite{AKS05} proved the following bound on the number of incidences between lines and reguli in 3-space. 

\begin{theorem}[Aronov et al.\cite{AKS05}]
\label{thm:linereginc}
Let $L$ be a set of $n$ lines in $\R^3$, and let $R$ be a set of $m$ reguli in $\R^3$. Then the number of incidences between the lines of $L$ and the reguli of $R$ is $O(n^{4/7} m^{17/21} + n^{2/3} m^{2/3} + m + n)$.
\end{theorem}

From this bound, one may derive an alternative proof of Theorem~\ref{4inreg}, of which we now give a brief sketch. First bound $P_{\geq k}$, defined as the number of reguli containing at least $k$ lines. From the above Theorem, we get $P_{\geq k}\lesssim n^3/k^{21/4}+n^2/k^3+n/k$. Then from Lemma~\ref{lem:nitoedges} we know that if no regulus contains more than $\ell$ lines, then the number of degenerate 4-tuples of lines is $m\lesssim n^3 + n^2\ell + n\ell^3$. Hence either $\ell$ is larger than $n^{2/3}$, or $m\lesssim n^3$ and from Theorem~\ref{thm:spencer} there exists an independent set of lines of size $\Omega (n^{1/3})$.

\acknowledgments
The authors wish to thank the reviewers for their comments, including those on earlier, preliminary versions of this paper.

\bibliographystyle{hplain}
\bibliography{independent_lines}
\appendix
\section{Proof of Lemma~\ref{le:mainle}}\label{app:cardproof} 

For the proof we need the following observation regarding generic projection maps.

\begin {lemma}
\label {le:tec} Let $P$ be a finite set of points in $\R^d$, and let
$\mathcal A$ be a finite set of $(d-2)$-flats in $\R^d$. 
Let $\pi$ be a generic projection from $\R^d$ to a hyperplane.
Then a point $p \in P$ lies on a $(d-2)$-flat $A \in \cal A$ if and only if $\pi(p) \in \pi(A)$.
\end {lemma}

\begin {proof}
The forward implication is clear. For the other direction, suppose $p \notin A$.
Then the affine span of $\{p\} \cup A$ is a hyperplane, that is, it is $(d-1)$-dimensional.
By the genericity of $\pi$, the image $\pi(\mbox{span}(\{p\} \cup A))$ must also be $(d-1)$-dimensional, so $\pi(p) \notin \pi(A)$.
\end {proof}

We also need the following result of Elekes and T\'oth~\cite{ET05}.
Given a point set $P$, a hyperplane $h$ is said to be \emph{$\gamma$-degenerate} if at most $\gamma |P \cap h|$ points of $P \cap h$ lie on a $(d-2)$-flat.
\begin{theorem}\label{thmET}
For every $d\geq 3$ there exist constants $C_d>0$ and $\gamma_d>0$ such that for every set of $n$ points in $\R^d$, the number $h_{\geq k}$ of $\gamma_d$-degenerate hyperplanes containing at least $k$ points of $P$ is at most
$$ C_d \left(\frac {n^d} {k^{d+1}} + \frac {n^{d-1}} {k^{d-1}}  \right).$$
\end{theorem}

For convenience we restate Lemma~\ref{le:mainle}.
\newtheorem*{repmainle}{Lemma \ref{le:mainle}}
\begin{repmainle}
Fix $d \ge 2$ and a set $\cal H$ of hyperplanes in $\R^d$. Let $P$ be a set of $n$ points in
$\R^d$ with no more than $\ell$ points in a hyperplane in $\cal H$, for some $\ell=O(n^{1/2})$. Then, the number of $(d+1)$-tuples in $P$ that lie
in a hyperplane in $\cal H$ is $O(n^d \log \ell)$.
\end{repmainle}

\begin{proof} 
The proof is an adaptation of the proof of Lemma 4.5 in Cardinal et al.~\cite{CTW14}.
It proceeds by induction on $d \ge 2$. The base case is $d=2$. We
wish to bound the number of triples of points of $P$, lying on a
line in $\cal H$. Let $h_k$ (resp., $h_{\ge k}$) denote the number
of lines of $\cal H$ containing exactly (resp., at least) $k$ points
of $P$. The number of triples of points lying on a line of $\cal H$
is

\begin {equation}
\begin {array}{ll}
\sum_{k=3}^\ell h_k \binom k 3 \le \sum_{k=3}^{\ell} k^2 h_{\ge k} \cr\cr
\lesssim \sum_{k=3}^{\ell} k^2\left(\frac {n^2} {k^3} + \frac n k\right)
\lesssim n^2 \log \ell + \ell^2 n \lesssim n^2 \log \ell,
\end {array}
\end {equation}
where $h_{\ge k} \lesssim \frac {n^2} {k^3} + \frac n k$ follows by the
Szemer\'edi-Trotter Theorem~\cite{ST83}.

We now consider the general case $d\ge 3$. Let $P$ be a set of $n$
points in $\R^d$, with no more than $\ell$ points in a hyperplane in
$\cal H$, where $\cal H$ is a given set of hyperplanes in $\R^d$,
and $\ell=O(n^{1/2})$.
Let $\gamma:= \gamma_d > 0$ be the constant specified in Theorem~\ref{thmET}.
We distinguish between the following three types of $(d+1)$-tuples:

\noindent {\bf Type 1: $(d+1)$-tuples of $P$ contained in a
$(d-2)$-flat in a hyperplane in $\cal H$.} 
Let $\cal F$ be the set of 
$(d-2)$-flats that are contained in some hyperplane in $\cal H$ and spanned by the points $P$.
Let $s_k$ denote the number of flats in $\cal F$ that contain exactly $k$ points of $P$.
We project $P$ onto a
$(d-1)$-flat $K$ via a generic projection $\pi$ to obtain a set of points
$P':=\pi(P)$ in $\R^{d-1}$.
Let $\cal H'$ be the set of hyperplanes $\pi(\Gamma)$ for each $\Gamma \in \cal F$.
By Lemma~\ref{le:tec}, $|P \cap \Gamma| = |P' \cap \pi(\Gamma)|$ for each $\Gamma \in \cal F$.
Thus $s_k$ is also the number of hyperplanes in $\cal H'$ containing $k$ points of $P'$.
Moreover, the hyperplanes in $\cal H'$ contain at most $\ell$ points of $P'$.

Applying the induction hypothesis on $P'$ with respect to $\cal H'$ 
we deduce that the number of $d$-tuples in $P'$ that lie in a
hyperplane in $\cal H'$ is
$$\sum_{k=d}^\ell s_k \binom k d \lesssim n^{d-1}\log \ell.$$
Therefore, the number of $(d+1)$-tuples of $P$ lying on a $(d-2)$-flat in $\cal F$ is
$$\sum_{k=d+1}^\ell s_k \binom k {d+1} \le \sum_{k=d+1}^\ell k s_k \binom k {d} \lesssim \ell n^{d-1} \log \ell \lesssim
n^d \log \ell.$$

\noindent {\bf Type 2: $(d+1)$-tuples of $P$ that span a
$\gamma$-degenerate hyperplane in $\cal H$.} 
Let $h_k$ denote the number of
$\gamma$-degenerate hyperplanes in $\cal H$ containing exactly $k$
points of $P$. 
Using Theorem~\ref{thmET}, we get

\begin {equation}
\begin {array}{ll}
\sum_{k=d+1}^\ell h_k \binom k {d+1} \le \sum_{k=d+1}^{\ell} k^d h_{\ge k}
\cr\cr \lesssim \sum_{k=d+1}^{\ell} k^d\left(\frac {n^d} {k^{d+1}} +
\frac {n^{d-1}} {k^{d-1}} \right) \lesssim n^d \log \ell + \ell^2 n^{d-1}
\lesssim n^d \log \ell .
\end {array}
\end {equation}

\noindent {\bf Type 3: $(d+1)$-tuples of $P$ that span a hyperplane
in $\cal H$ that is not $\gamma$-degenerate.} 
Recall that if a hyperplane $H$
spanned by $P$ is not $\gamma$-degenerate, then more than a $\gamma$
fraction of its points lie in some $(d-2)$-flat. 
Consider a $(d-2)$-flat $L$ 
containing exactly $k$ points of $P$. 
A point in $P \setminus L$ can be on at most one hyperplane containing $L$.
Let $n_r$
denote the number of hyperplanes in $\cal H$ containing $L$ and
exactly $r$ points of $P\setminus L$. Then $\sum_r n_r r \le n$, and by assumption on the
hyperplanes in $\cal H$, we have $r \le \ell$. 

We will assign each tuple of Type 3 to a $(d-2)$-flat that causes it to be Type 3. 
Fix a $(d-2)$-flat $L$ with $k$ points and consider a hyperplane $H \in \cal H$ that is not $\gamma$-degenerate because it contains $L$. 
That is, suppose $H$ contains $r+k$ points, and $k > \gamma(r+k)$, so $r< O(k)$.
All tuples that span H contain at least one point not in $L$. 
Hence the number of tuples that span $H$ is $O(rk^d)$. 
Assign these tuples to $L$. 
The total number of tuples of Type 3 that will be assigned to $L$ in this way is therefore at most
$$O \left( \sum_r n_r r k^d \right) \lesssim nk^d.$$

Let $\cal F$ be the set of $(d-2)$-flats that 
have at least one Type 3 tuple assigned to them.
Thus $\cal F$ is a finite set.
Let $s_k$ denote the number of flats in $\cal F$ that contain exactly $k$ points of $P$.
We project $P$ onto a
$(d-1)$-flat $K$ via a generic projection $\pi$ to obtain a set of points
$P':=\pi(P)$ in $\R^{d-1}$.
Let $\cal H'$ be the set of hyperplanes $\pi(\Gamma)$ for each $\Gamma \in \cal F$.
By Lemma~\ref{le:tec}, $|P \cap \Gamma| = |P' \cap \pi(\Gamma)|$ for each $\Gamma \in \cal F$.
Thus $s_k$ is also the number of hyperplanes in $\cal H'$ containing $k$ points of $P'$.
Moreover, the hyperplanes in $\cal H'$ contain at most $\ell$ points of $P'$.
Applying the induction hypothesis on $P'$ with respect to $\cal H'$ 
we deduce that the number of $d$-tuples in $P'$ that lie in a
hyperplane in $\cal H'$ is
$$\sum_{k=d}^\ell s_k \binom k d 
\lesssim n^{d-1}\log \ell.$$
Moreover, $\sum_{k=1}^{d-1} s_k k^d \lesssim n^{d-1}$.
Therefore, the number of $(d+1)$-tuples of Type 3 is at most
$$\sum_{k=1}^\ell s_k n k^d \le n\sum_{k=1}^\ell s_k k^d \lesssim n^d \log \ell.$$

Summing over all three cases, the proof is complete.
\end {proof}
\end{document}